\documentclass[letterpaper, two column, 10 pt, conference]{ieeeconf}  

\IEEEoverridecommandlockouts                              
\usepackage[utf8]{inputenc}
\usepackage{amsmath,bm}
\usepackage{amsfonts}
\usepackage{amssymb}
\usepackage{graphicx}
\usepackage{amsfonts}
\usepackage[colorlinks,urlcolor=blue]{hyperref}
\usepackage{caption}
\usepackage{subcaption}
\usepackage{algorithm}
\usepackage{algpseudocode}
\usepackage{cases}
\usepackage{enumerate}
\usepackage{xcolor}
\usepackage{cite}
\usepackage{verbatim}
\usepackage[normalem]{ulem}

\newtheorem{definition}{Definition}
\newtheorem{Theorem}{Theorem}

\newtheorem{Lemma}{Lemma}

\newtheorem{Corollary}{Corollary}

\usepackage{enumerate}
\usepackage{amsmath,bm}
\usepackage{multirow}
\newcommand{\m}[1]{\mathbf{#1}}
\newcommand{\mc}[1]{\mathcal{#1}}
\newcommand{\mb}[1]{\mathbb{#1}}
\newcommand{\abs}[1]{\lVert{#1} \rVert}

\newcommand\numberthis{\addtocounter{equation}{1}\tag{\theequation}}
\usepackage{tikz}
	\usetikzlibrary{arrows,positioning,bending,shapes,backgrounds,fit}  
	\tikzstyle{frame} = [draw, -latex]
	\tikzstyle{line} = [draw]
	\tikzstyle{line2} = [draw, dashdotted]
	\tikzstyle{line3} = [draw, dashed]
	\tikzstyle{line3UD} = [draw, dashed]
	\tikzstyle{place} = [circle, draw=black, fill=white, thick, inner sep=2pt, minimum size=1mm]
	\tikzstyle{place2} = [circle, draw=black, fill=black, thick, inner sep=2pt, minimum size=1mm]
	\tikzstyle{placeRed} = [circle, draw=red, fill=red, thick, inner sep=2pt, minimum size=1mm]
	\tikzstyle{vertex} = [circle, draw=black, fill=black, thick, inner sep=2pt, minimum size=1mm]
\usepackage{tkz-euclide}
\makeatletter
\def\endthebibliography{%
  \def\@noitemerr{\@latex@warning{Empty `thebibliography' environment}}%
  \endlist
}
\makeatother

\title{\LARGE Further Analysis on Structure and Spectral Properties of Symmetric Graphs}

\author{Quoc Van Tran$^\dag$ and Hyo-Sung Ahn$^\ddag$
\thanks{$^\dag$ Department of Mechanical Engineering, Korea Advanced Institute of Science and Technology (KAIST),
        Daejeon, Republic of Korea.
        E-mails: {\tt\small quoctran@kaist.ac.kr; tvquoc9790@gmail.com}}
\thanks{$^\ddag$ School of Mechanical Engineering,
        Gwangju Institute of Science and Technology, Gwangju, Republic of Korea.
        E-mail: {\tt\small hyosung@gist.ac.kr}}
}
\begin{document}
\maketitle
\begin{abstract}
Graph is an abstract representation commonly used to model networked systems and structure. In problems across various fields, including computer vision and pattern recognition, and neuroscience, graphs are often brought into comparison (a process is called \textit{graph matching}) or checked for symmetry. Friendliness property of the associated adjacency matrices, specified by their spectral properties, is important in deriving a convex relaxation of the (intractable) discrete graph matching problem. In this work, we study unfriendliness properties of symmetric graphs by studying its relation to the underlying graph structure. It is revealed that a symmetric graph has two or more subgraphs of the same topology, and are adjacent to the same set of vertices. We then show that if adjacency matrices of symmetric graphs have distinct eigenvalues then there exist eigenvectors orthogonal to the vector of all ones, making them unfriendly. Relation of graph symmetry to uncontrollability of multi-agent systems under agreement dynamics with one controlled node is revisited.
Examples of both synthetic and real-world graphs are also given for illustrations.
\end{abstract}
\section{Introduction}
Graph is an abstract representation commonly used to model networked dynamical systems \cite{Friedkin2016Sci, Quoc2018tcns, Bullo2020, Quoc2020tcns, Mesbahi2010, HSAhn2019, Quoc2020auto, Quoc2021tcns} and structure \cite{DefferrardNIPS16,SeifertTSP21}. In problems across various fields, including computer vision and pattern recognition \cite{Foggia2014}, neuroscience \cite{Vogelstein2015PloSONE}, formation control and energy networks \cite{QuocCDC20}, graphs are often brought into comparison.
The adjacency matrix of a friendly graph on $n$ vertices, $\bm{A}\in \{0,1\}^{n\times n}$ has distinct eigenvalues and eigenvector not orthogonal to the vector of all ones, denoted as $\bm{1}_n\in \mb{R}^n$ \cite{Aflalo2015pnas} (Definition \ref{def:friendly_graphs}). A graph $\mc{G}$ is \textit{symmetric} if there exists a permutation matrix $\bm{\Pi}\neq \bm{I}_n$ so that $\bm{\Pi}\bm{A}(\mc{G})=\bm{A}(\mc{G})\bm{\Pi}$. In other words, the vertex permutation representing $\bm{\Pi}$ acting on the graph $\mc{G}$ leaves its graph structure unchanged. If there is no such an permutation matrix $\bm{\Pi}$, we say the graph $\mc{G}$ \textit{asymmetric} \cite{Erdos1963}. The work \cite{Aflalo2015pnas} shows that friendly graphs have asymmetric structures.

Consider two isomorphic graphs $\mc{G}_1$ and $\mc{G}_2$ on $n$ vertices with adjacency matrices being $\bm{A}$ and $\bm{B}$, respectively. \textit{Graph matching (GM)} asks to find a vertex correspondence associating vertices in two graphs, or equivalently a permutation matrix $\bm{\Pi}$ satisfying $\bm{\Pi}\bm{A}=\bm{B}\bm{\Pi}$ \cite{Foggia2014,Aflalo2015pnas}.  If the two graphs to be matched are symmetric, there are two or more $\bm{\Pi}$ that satisfy the preceding relation. Moreover, since the set of permutation matrices, say $\mc{P}_n$, has $n!$ elements, the discrete graph matching problem becomes intractable as $n$ increases. Heuristics have thus been proposed to approximately compute graph matching without theoretical guarantee of obtaining the true permutation matrix \cite{Foggia2014, Ullmann2011}. Spectral methods inspect graph matching by reliance on the similarity between the spectral properties of the adjacency matrices $\m{A}$ and $\m{B}$ \cite{Caelli2004pami, Duchenne2011pami,ZFan2020icml}. To efficiently solve the discrete GM problem, continuous relaxations of GM \cite{Aflalo2015pnas, Vogelstein2015PloSONE, Lyzinski2016pami,QuocCDC20} are used by the replacement of the permutation matrix set with a larger set of doubly stochastic matrices $\mc{D}:=\{\bm{P}\in \mb{R}^{n\times n}:\bm{P}\bm{1}_n=\bm{1}_n,~\bm{1}_n^\top\bm{P}=\bm{1}_n^\top\}$.
Convex relaxation of GM, which minimizes the adjacency disagreement $\abs{\bm{\Pi}\bm{A}-\bm{B}\bm{\Pi}}_F^2$ over the set $\mc{D}$, is often considered as being a convex problem.
For this relaxation, friendliness properties (we formally define this notion in Definition \ref{def:friendly_graphs}) of the adjacency matrices of graphs are crucial in ensuring the convergence, e.g., occurred using a gradient-based algorithm, to the actual graph matching of two isomorphic graphs \cite{Aflalo2015pnas}. Distributed computation for graph matching of two asymmetric graphs with real edge weights is presented in \cite{QuocCDC20}. 
Uncontrollability of multi-agent systems is also closely related to the symmetry of the graph of the system with respect to the control node\cite{Rahmani2009siam}. 

In this paper, we investigate further the unfriendliness properties of the adjacency matrices of symmetric graphs by exploring their relation to the symmetry of the underlying graph topologies.
The first contribution of this paper lies in the identification of the spectral properties of permutation matrices by inspection of their corresponding permutation graph representations that have cycles of even lengths. Second, we elaborate further symmetric structures of graphs in association with the the permutation graphs of permutation matrices. It is revealed that a symmetric graph has two or more subgraphs of the same graph structure, which are adjacent to the same set of vertices in the rest of the graph. Empirical experiments in random graphs suggest that as $n$ increases and the edges in the graph are not too dense nor sparse, its adjacency matrix has simple spectrum. We then show that due to the symmetric structure of symmetric graphs their adjacency matrices have an eigenvector orthogonal to $\bm{1}_n$, thus possessing unfriendliness. Third, relation of graph symmetry to uncontrollability of multi-agent systems under agreement dynamics with one controlled node is presented. Suppose that the graph of the other floating nodes is symmetric. Then, if the controlled node is adjacent to both vertices in some  pairs of corresponding vertices in two subgraphs of symmetry then the system is uncontrollable.

The remainder of this paper is as follows. Preliminaries are given in Section \ref{sec:prel}. Section \ref{sec:main} presents main observations on spectral properties of permutation matrices, and structures and unfriendly properties of symmetric graphs. We elaborate further the relation of graph symmetry to uncontrollability of multi-agent systems in Section \ref{sec:relation_to_uncontroll}. Several examples are provided in Section \ref{sec:examples} and Section \ref{sec:conclusion} concludes this paper.
\section{Preliminaries}\label{sec:prel}
\subsection{Permutation graphs}
Let the set of $n\times n$ permutation matrices be $\mc{P}_n$. Consider a set of vertices $\mc{V}=\{1,\ldots,n\}$.
We define $\pi:\mc{V}\rightarrow \mc{V}$ as the \textit{vertex mapping} associated with a permutation matrix $\bm{\Pi}\in \mc{P}_n$. In particular,  $\pi$ maps a vertex $i\in\mc{V}$ to a vertex $\pi(i)$ in $\mc{V}$ itself, which can represented in the two-line form
\begin{equation}
\left( \begin{matrix}
1 & 2 &\ldots &n\\
\pi(1) & \pi(2) &\ldots &\pi(n)
\end{matrix}\right).
\end{equation}
The permutation graph ${G}_\pi=(\mc{V},\mc{E}_\pi)$ associated with $\pi$ can be defined as follows. ${G}_\pi$ contains $n$ nodes indexed in $\mc{V}$. A directed edge $(i,j)\in \mc{E}_\pi$ if $\pi(i)=j$, i.e., node $i$ takes the position of $j$ under $\pi$. For example, for the permutation \begin{equation}\pi= \begin{bmatrix}
1 & 2&3 &4 &5\\
2 &1 &4 &5 &3
\end{bmatrix},\end{equation} the corresponding permutation matrix and permutation graph $G_\pi$ are respectively given as in Fig. \ref{fig:permutation_graph}.

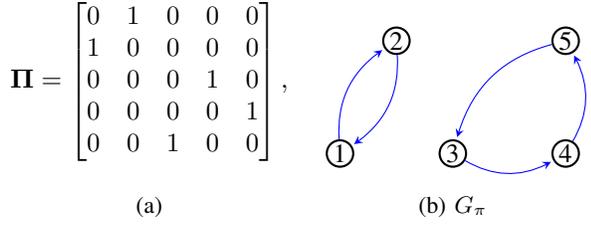
\begin{figure}[t]
\centering
\begin{subfigure}[b]{0.22\textwidth}
\centering
\begin{align*}
\bm{\Pi} = \left[\begin{matrix}
0 &1 &0 &0 &0 \\
1 &0 &0 &0 &0\\
0 &0 &0 &1 &0\\
0 &0 &0 &0 &1\\
0 &0 &1 &0 &0
\end{matrix}\right],
\end{align*}
\caption{}
\end{subfigure}
\begin{subfigure}[b]{0.22\textwidth}
\centering
\begin{tikzpicture}
 [>=stealth,
   shorten >=1pt,
   node distance=2cm,
   on grid,
   auto,
   every state/.style={draw=black, fill=black!6, very thick},scale =1.5
  ]
\tikzstyle{neuron}=[place,circle,inner sep=0,minimum size=10pt]
\node[neuron] (4) at (1,0.) [label=right:] {4};
\node[neuron] (3) at (0,0) [label=left:] {3};
\node[neuron] (1) at (-1.,0) [label=left:] {1};
\node[neuron] (2) at (-.5,1.) [] {2};
\node[neuron] (5) at (1.,1.) [] {5};
\path[->] 
   (1) edge[bend left,blue]     node                   {} (2)
   (2) edge[bend left,blue]     node                   {} (1)
   (3) edge[bend right,blue]     node                   {} (4)
   (4) edge[bend right,blue]     node                   {} (5)
   (5) edge[bend right,blue]     node                   {} (3);
\end{tikzpicture}
\caption{$G_\pi$}
\end{subfigure}
\caption{The permutation matrix $\bm{\Pi}$ and permutation graph $G_\pi$ representing the vertex mapping $\pi$.}
\label{fig:permutation_graph}
\end{figure}
Note that the $i$-th column of the identity matrix $\bm{I}_5$ appears in the $\pi(i)$-th column of $\bm{\Pi}$. The permutation graph $G_\pi$ contains two directed cycles $\mc{C}_\pi^1=\{1\rightarrow 2\rightarrow 1\}$ and $\mc{C}_\pi^2=\{3\rightarrow 4\rightarrow 5\rightarrow 3\}$, where the symbol $\rightarrow$ denotes the direction of the directed edges.
Further, since each node in a permutation graph $G_\pi$ has precisely in-degree and out-degree of $1$, one can easily verify the following lemma.
\begin{Lemma}\label{lm:independent_cycles}
Each connected component in the permutation graph $G_\pi$ must be a (directed) cycle, and hence any permutation can be expressed as a combination of independent cycles $G_\pi=\{\mc{C}_{\pi}^1,\mc{C}_{\pi}^2,\ldots\}$.
\end{Lemma} 
\subsection{Matching of undirected graphs}
Consider two undirected graphs without self-loops on $n$ vertices $\mathcal{G}_1=(\mathcal{V},\mathcal{E}_1,\bm{A})$ and $\mc{G}_2=(\mc{V},\mc{E}_2,\bm{B})$ having a common vertex set $\mc{V}=\{1,\ldots,n\}$ and edge sets $\mc{E}_1,\mc{E}_2\subseteq \mc{V}\times \mc{V}$, respectively. If an edge $(i,j)\in \mc{E}_1$ then $i$ and $j$ are said to be adjacent to each other, and we simply write $i\sim j$. The set of vertices adjacent to a vertex $i$ in $\mc{G}_1$ is denoted as $\mc{N}_i(\mc{G}_1)=\{j\in \mc{V}:(i,j)\in \mc{E}_1\}$. We shall use $\mc{E}(\mc{G}_1)$ and $\mc{E}_1$ interchangeably to denote the edge set of the graph $\mc{G}_1$.

We define $\bm{A}=[a_{ij}]$ and $\bm{B}=[b_{ij}]\in \{0,1\}^{n\times n}$ as the (symmetric) adjacency matrices of the graphs $\mc{G}_1$ and $\mc{G}_2$, respectively. In particular, the weight $a_{ij}=1$ (resp. $b_{ij}=1$) if the edge $(i,j)\in \mc{E}_1$ (resp. $(i,j)\in \mc{E}_2$) and $a_{ij}=0$ (resp. $b_{ij}=0$), otherwise.
One wants to find such a vertex mapping $\pi$, which associates a vertex $i$ in $\mc{G}_1$ to a vertex $\pi(i)$ in $\mc{G}_2$, and similarly corresponds each entry $(\bm{A})_{ij}$ to an entry $(\bm{B})_{\pi(i)\pi(j)}$. That is, each edge $(i,j)$ in $\mc{G}_1$ corresponds to an edge $(\pi(i),\pi(j))$ in $\mc{G}_2$. Further, in the case of exact matching, using the adjacency matrices, one has $\bm{A}=\bm{\Pi}^\top\bm{B}\bm{\Pi}$ \cite{Aflalo2015pnas}. Thus, we measure the \textit{adjacency disagreement} between two adjacency matrices $\bm{A}$ and $\bm{B}$ by the following \textit{distortion function}
\begin{equation}\label{eq:distortion_function}
\mathrm{dis}_{\mc{G}_1\rightarrow\mc{G}_2}(\bm{\Pi})=||\bm{A}-\bm{\Pi}^\top\bm{B}\bm{\Pi}||_F=||\bm{\Pi}\bm{A}-\bm{B}\bm{\Pi}||_F.
\end{equation}
If $\mathrm{dis}_{\mc{G}_1\rightarrow\mc{G}_2}(\bm{\Pi})=0$ for some $\bm{\Pi}\in \mc{P}_n$, then two graphs $\mc{G}_1$ and $\mc{G}_2$ are said to be \textit{isomorphic}.
Let $\mathrm{Iso}(\mc{G}_1\rightarrow\mc{G}_2)=\{\bm{\Pi}\in \mc{P}_n:\mathrm{dis}_{\mc{G}_1\rightarrow\mc{G}_2}(\bm{\Pi})=0\}$ be the collection of all \textit{isomorphisms} associating $\mc{G}_1$ and $\mc{G}_2$. Obviously, $\mc{G}_1$ is isomorphic to itself via the identity permutation $\bm{I}_n$. 

\subsection{Symmetric graphs}
The \textit{automorphism group} of the graph $\mc{G}$ is denoted as $\mathrm{Auto}(\mc{G}):=\mathrm{Iso}(\mc{G}\rightarrow\mc{G})\supseteq \{\bm{I}_n\}$.
A graph $\mc{G}=(\mc{V},\mc{E})$ is said to be \textit{symmetric}\footnote{The symmetry or asymmetry of a graph should be distinguished from the symmetry of its associated adjacency matrix. The latter itself is \textit{symmetric} simply when the graph $\mc{G}$ is undirected.} if it has a nontrivial automorphism group, i.e., $\mathrm{Auto}(\mc{G})\setminus \{\bm{I}_n\}\neq \emptyset$. That is, there exists a nonidentity $\bm{\Pi}\in \mathrm{Auto}(\mc{G})$ such that $\bm{\Pi}\bm{A}(\mc{G})-\bm{A}(\mc{G})\bm{\Pi}=\bm{0}$, where $\bm{A}(\mc{G})$ is the adjacency matrix of $\mc{G}$. In contrast, the graph $\mc{G}$ is \textit{asymmetric} if $\mathrm{Iso}(\mc{G}\rightarrow\mc{G})= \{\bm{I}_n\}$ \cite{Erdos1963}.

The degree, $d_i$, of a node $i$ is equal to the number of its neighbors, i.e., $d_i=|\mc{N}_i|$. Let $\bm{D}(\mc{G})=\mathrm{diag}(d_i)$ be the \textit{degree matrix} of the graph $\mc{G}$. Then, the Laplacian matrix of $\mc{G}$ is given as $\bm{L}=\bm{D}-\bm{A}$. Note importantly that the degree of a vertex is invariant under an automorphic mapping, i.e., $d_i=d_{\pi(i)}$ for any $i\in \mc{V}$ and $\bm{\Pi}\in \mathrm{Auto}(\mc{G})$. Thus, a nonidentity permutation $\bm{\Pi}\in \mathrm{Auto}(\mc{G})$ if the following holds
\begin{equation}\label{eq:sym_cond_laplacian}
\bm{\Pi}\bm{L}(\mc{G})=\bm{L}(\mc{G})\bm{\Pi}.
\end{equation}

We now define a class of \textit{friendly} graphs, as mentioned in the introduction.
\begin{definition}[Friendly Graphs]\cite{Aflalo2015pnas}\label{def:friendly_graphs}
A graph $\mc{G}$ is said to be \textit{friendly} if its adjacency matrix $\bm{A}(\mc{G})$ has simple spectrum and no eigenvector $\bm{v}_i\in \mb{R}^n$ orthogonal to $\bm{1}_n$.
\end{definition}
\begin{figure}
\centering
\includegraphics[height = 7.5cm]{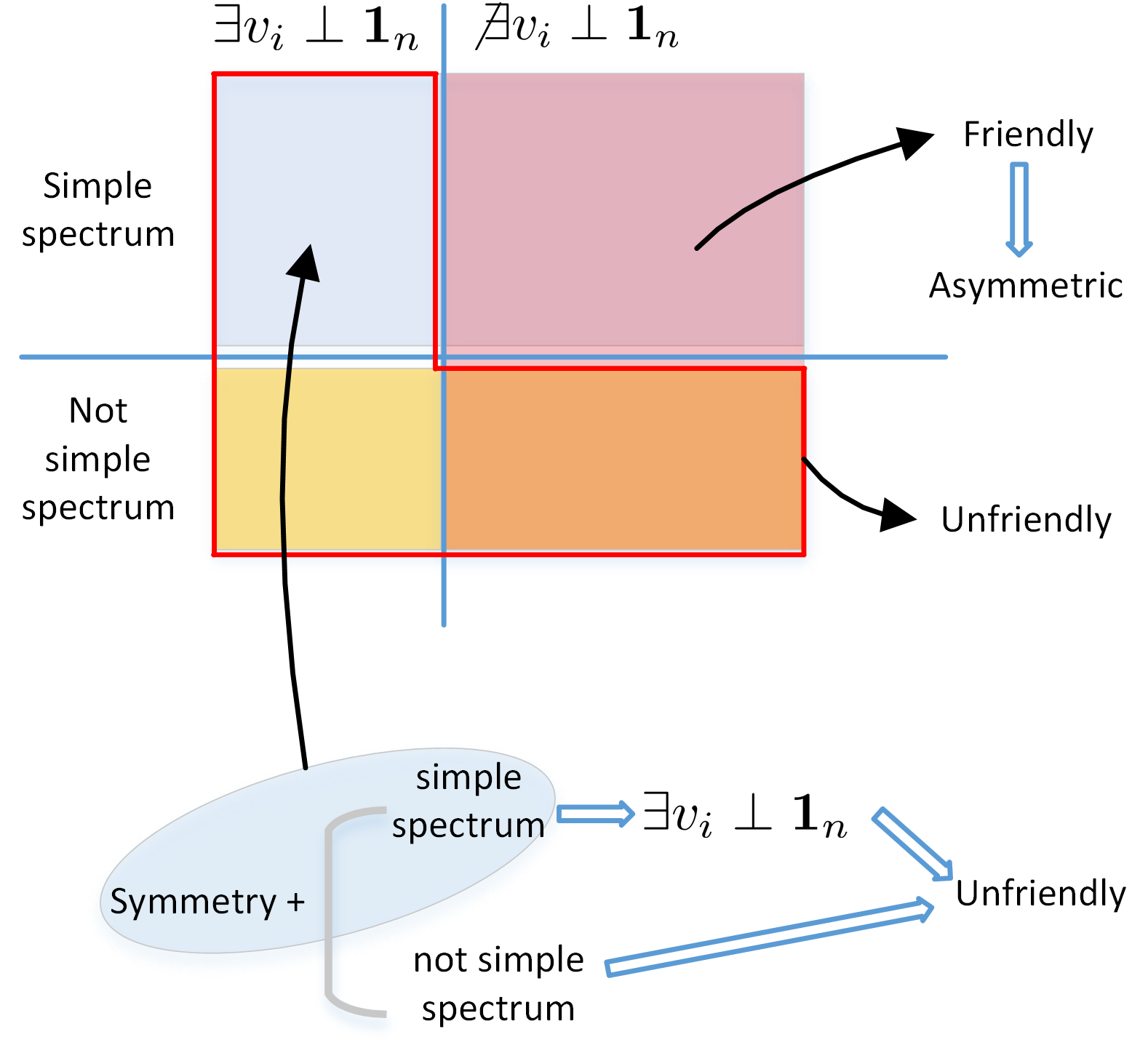}
\caption{Adjacency matrices of friendly graphs have simple spectrum and no eigenvector $\bm{v}_i\perp \bm{1}_n$. Friendly graphs are asymmetric. Though not all asymmetric graphs are friendly, for a large $n$, almost all graphs are asymmetric. Given a symmetric graph, if it has distinct eigenvalues then there exists an eigenvector orthogonal to $\bm{1}_n$.}
\label{fig:diagram}
\end{figure}
Graphs that are not friendly, i.e., having either repeated eigenvalues or an eigenvector, namely $\bm{v}_i$, orthogonal to $\bm{1}_n$, i.e., $\bm{v}_i^\top\bm{1}_n=0$, will be called \textit{unfriendly}. We explain further the notions of friendliness (light pink area) versus unfriendliness (the portions inside the red lines) using the diagram in Fig. \ref{fig:diagram}. It is shown in \cite[Lemma 1]{Aflalo2015pnas} that friendly graphs have asymmetric structures.  Furthermore, unfriendliness properties appear to be non-generic. This is in fact true for graphs with non-negative real edge weights \cite{Quoc2022Tcyb}. Therefore, one expects that the light pink region in Fig. \ref{fig:diagram} covers almost all the space.

In the following subsection, we empirically examine the possibility of random graphs possessing unfriendly properties (see \cite{Tao2017} for more theoretical analysis).
\subsection{Possibility of random Erd\"os-R\'enyi graphs possessing unfriendly properties}
Consider a class of random Erd\"os-R\'enyi graphs $\mc{G}(n,p)$ on $n$ vertices. Any two vertices $i$ and $j$ are randomly and independently connected by an edge with probability $p\in [0,1]$ \cite{Hofstad2016}. For each number of vertices $n$ from $2$ to $100$ and probabilities $p\in [0,1]$, $5000$ random graphs $\mc{G}(n,p)$ are created. We then empirically compute the possibilities of the adjacency matrix of $\mc{G}(n,p)$ containing repeated eigenvalues (Fig. \ref{fig:prob_repeated_eigenvalues}) and having an eigenvector orthogonal to $\bm{1}_n$ (Fig. \ref{fig:prob_orthog_eigenvectors}) with tolerance level of $10^{-4}$, respectively. It is observed that for large graphs ($n\geq 15$) with not too low or high edge possibilities ($0.1\leq p \leq 0.9$), their adjacency matrices are likely to have distinct eigenvalues. Further, the possibility of these adjacency matrices having an eigenvector orthogonal to $\bm{1}_n$ is relatively low (less than $1\%$).

\begin{figure}[t]
\centering
\begin{subfigure}{0.43\textwidth}
\centering
\includegraphics[height = 5cm]{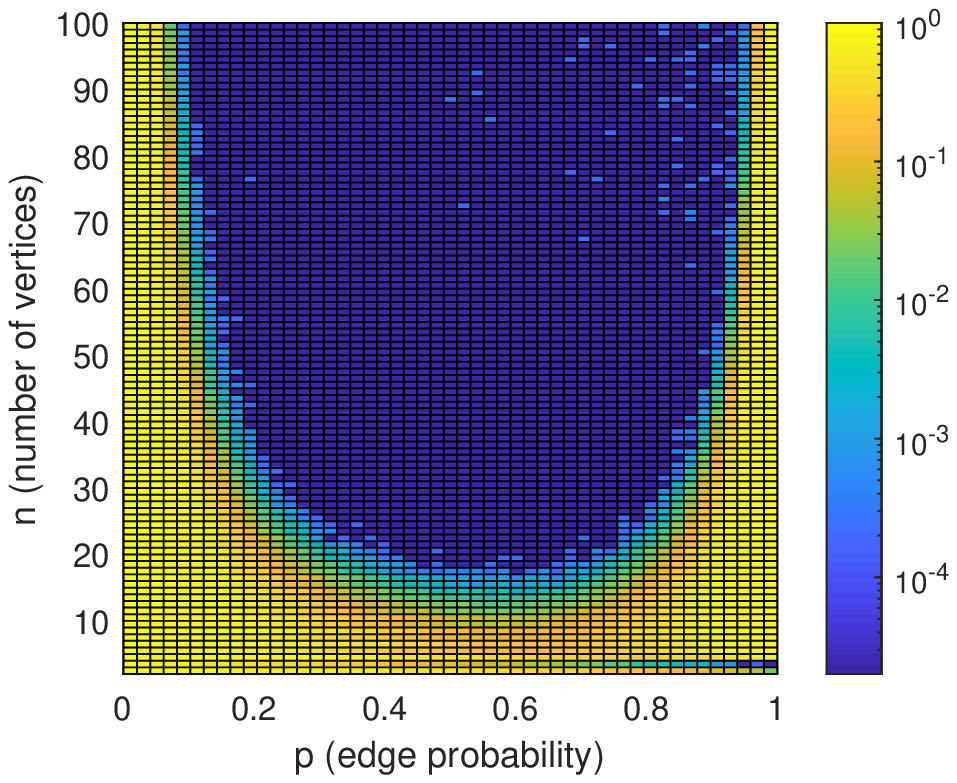}
\caption{}
\label{fig:prob_repeated_eigenvalues}
\end{subfigure}
\begin{subfigure}{0.43\textwidth}
\centering
\includegraphics[height = 5cm]{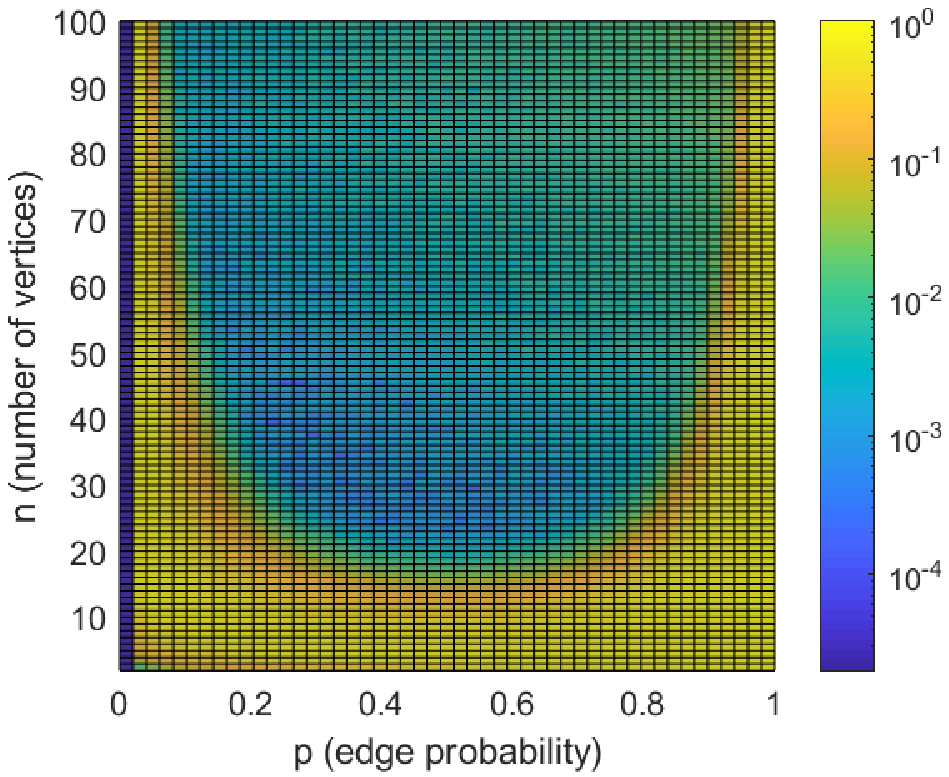}
\caption{}
\label{fig:prob_orthog_eigenvectors}
\end{subfigure}
\caption{Probability of the unweighted adjacency matrix of $\mc{G}(n,p)$ (a) containing repeated eigenvalues, (b) having an eigenvector orthogonal to $\bm{1}_n$. Probabilities are averaged over $5000$ experiments.} 
\label{fig:prob_unfriendliness}
\end{figure}

\section{Structure and unfriendliness of symmetric graphs}\label{sec:main}
This section first explores symmetric structures of symmetric graphs. We then elaborate further that if the adjacency matrix of a symmetric graph has simple spectrum, then it ought to contain an eigenvector orthogonal to $\bm{1}_n$, making the graph unfriendly (see Fig. \ref{fig:diagram}).
\subsection{Structure of symmetric graphs}
We first verify the existence of an eigenvalue of $-1$ of permutation matrices whose permutation graphs contain even length cycles.
\begin{Lemma}\label{lm:minus_one_eigenvalue}
A permutation matrix $\bm{\Pi}$ has an eigenvalue of $-1$ if and only if its permutation graph has an even length cycle.
\end{Lemma}
\begin{proof}
We show that it is always possible to construct an eigenvector $\bm{v}\in \mb{R}^n$ satisfying $\bm{\Pi}\bm{v}=-\bm{v}$, if the permutation graph $G_{\bm{\pi}}$ has a (directed) cycle, $\mc{C}_{\pi}^1$, of even length, say $l=2,4,\ldots~(l\leq n)$. Without loss of generality, we relabel the vertices in $\mc{V}$ so that the vertices and edges of $\mc{C}_{\pi}^1$ are $\mc{V}(\mc{C}_{\pi}^1)=\{1,2,\ldots,l\}$ and $\mc{E}(\mc{C}_{\pi}^1)=\{1\rightarrow 2,2\rightarrow 3,\ldots,l\rightarrow 1\}$, respectively. Now, we consider a vector of the form 
\begin{equation}
\bm{v}=[\underbrace{1,-1,1,\ldots,-1}_{l},\underbrace{0,\ldots,0}_{n-l}]^\top.
\end{equation}
It is easy to see that under the permutation cycle $\mc{C}_{\pi}^1$ we have $\bm{\Pi}\bm{v}=[-1,1,-1,\ldots,1,0,\ldots,0]^\top=-\bm{v}$. Thus, $\bm{v}$ is an eigenvector of $\bm{\Pi}$ corresponding to the eigenvalue $-1$. 

On the other hand, if $G_\pi$ has only odd length cycles\footnote{We will use the convention that the mapping of a vertex to itself is considered as a permutation cycle of length zero.} we shall show that it is not possible to construct such a $\bm{v}\not\equiv \bm{0}$ so that $\bm{\Pi}\bm{v}=-\bm{v}$. Indeed, consider an arbitrary cycle $\mc{C}_\pi^1$ of odd length $l$ and a numbering of the vertices in $\mc{V}$ so that $\mc{V}(\mc{C}_\pi^1)=\{1,2,\ldots,l\}$ and $\mc{E}(\mc{C}_\pi^1)=\{1\rightarrow 2,2\rightarrow 3,\ldots,l\rightarrow 1\}$. Then, for any vector $\bm{v}=[v_1,\ldots,v_n]^\top\in \mb{C}^n$, the relation $\bm{\Pi}\bm{v}=-\bm{v}$ leads to 
\begin{equation}
v_1=-v_2=v_3=\ldots=v_l=-v_1
\end{equation} due to the permutation under $\mc{C}_\pi^1$. It follows that $v_1=v_2=\ldots=v_l =0$. Since all cycles in $G_\pi$ have odd lengths, following the preceding argument one has $\bm{v}\equiv \bm{0}$, which is a contradiction.
\end{proof}
\begin{Corollary}\label{coroll:number_of_eig}
The number of eigenvalues $-1$ of a permutation matrix $\bm{\Pi}$ is equal to the number of independent even length cycles in its permutation graph $G_\pi$.
\end{Corollary}

We now show the symmetric structure of symmetric graphs and the existence of cycles of length $2$ in the permutation graph associated with $\bm{\Pi}\in \mathrm{Auto}(\mc{G})$.
\begin{Theorem}\label{thm:existence_cycle_length2}
Let $\mc{G}$ be a symmetric graph. Then, there exists a nontrivial $\bm{\Pi}\in \mathrm{Auto}(\mc{G})$ such that the corresponding permutation graph $G_\pi$ contains a cycle of length $2$.
\end{Theorem}
\begin{proof}
Since $\mc{G}$ is symmetric there exist some nontrivial $\bm{\Pi}_1\in \mathrm{Auto}(\mc{G})$. Suppose that the permutation graph $G_{\pi_1}$ corresponding to $\bm{\Pi}_1$, without loss of generality, has a cycle of length $3$ (or higher), e.g., \begin{equation}
\mc{C}_{\pi_1}^1=\left(\{i,j,k\},\{i\rightarrow j,j\rightarrow k,k\rightarrow i\}\right).
\end{equation}
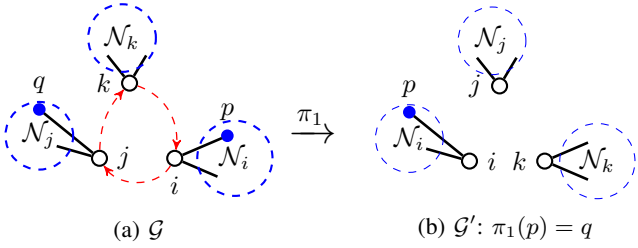
\begin{figure}[t]
\centering
\begin{subfigure}[h]{0.21\textwidth}
\centering
\begin{tikzpicture}[scale=1,>=stealth',auto,node distance=3cm,
  thick,main node/.style={circle,draw,font=\sffamily\Large\bfseries}]
\node[place] (i) at (1,0.) [label=below:$i$] {};
\node[place, blue, scale = 0.7] (i1) at (1.7,0.3) [label=above:$p$] {};
\node[] (i2) at (1.7,-0.3) [] {};
\node[place] (j) at (0,0) [label=right:$j$] {};
\node[place, blue, scale = 0.7] (j1) at (-0.8,0.65) [label=above:$q$] {};
\node[] (j2) at (-0.7,0.2) [] {};
\node[place] (k) at (0.4,1.) [label=left:$k$] {};
\node[] (k1) at (0.7,1.5) [] {};
\node[] (k2) at (0,1.5) [] {};
\draw (i) [line width=1pt] -- node [left] {} (i1) 
	  (i) --node [left] {}(i2);
\draw[blue,dashed] (1.8,0.) circle (0.5) node[black]{$\mc{N}_i$};
\draw (k) [line width=1pt] -- node [left] {} (k1) (k)--(k2);
\draw[blue,dashed] (0.3,1.6) circle (0.45) node[black]{$\mc{N}_k$};
\draw (j) [line width=1pt] -- node [left] {} (j1) (j)--(j2);
\draw[blue,dashed] (-0.8,0.3) circle (0.45) node[black]{$\mc{N}_j$};
\path[every node/.style={font=\large},->, dashed, red, line width=0.5pt]
    (i) edge[bend left = 60] node [left] {} (j)
    (j) edge[bend left = 20] node [left] {} (k)
    (k) edge[bend left = 40] node [left] {} (i);
\end{tikzpicture}
\caption{$\mc{G}$}
\label{fig:three_vertices_w_neighbors}
\end{subfigure} 
\begin{subfigure}[h]{0.05\textwidth}\begin{Large}$\xrightarrow{\pi_1}$\end{Large}
\end{subfigure} 
\begin{subfigure}[h]{0.21\textwidth}
\centering
\begin{tikzpicture}[scale=1]
\node[place] (i) at (1,0.) [label=left:$k$] {};
\node[] (i1) at (1.7,0.3) [] {};
\node[] (i2) at (1.7,-0.3) [] {};
\node[place] (j) at (0,0) [label=right:$i$] {};
\node[place, blue, scale = 0.7] (j1) at (-0.8,0.65) [label=above:$p$] {};
\node[] (j2) at (-0.7,0.2) [] {};
\node[place] (k) at (0.4,1.) [label=left:$j$] {};
\node[] (k1) at (0.7,1.5) [] {};
\node[] (k2) at (0,1.5) [] {};
\draw (i) [line width=1pt] -- node [left] {} (i1) (i)--(i2);
\draw[blue,dashed] (1.7,0.) circle (0.5) node[black]{$\mc{N}_k$};
\draw (k) [line width=1pt] -- node [left] {} (k1) (k)--(k2);
\draw[blue,dashed] (0.3,1.6) circle (0.45) node[black]{$\mc{N}_j$};
\draw (j) [line width=1pt] -- node [left] {} (j1) (j)--(j2);
\draw[blue,dashed] (-0.8,0.3) circle (0.45) node[black]{$\mc{N}_i$};
\end{tikzpicture}
\caption{$\mc{G}^\prime$: $\pi_1(p)=q$}
\end{subfigure}
\caption{Vertices $\{i,j,k\}$ and their neighbor sets in $\mc{G}$. The graph $\mc{G}^\prime$ is obtained from $\mc{G}$ via the permutation $\pi_1$ (red directed cycle). The topologies of $\mc{G}^\prime$ and $\mc{G}$ are the same.}
\label{fig:three_vertices}
\end{figure}
The vertices $\{i,j,k\}$ and their corresponding (implicit) neighbor sets in $\mc{G}$ are depicted in Fig. \ref{fig:three_vertices_w_neighbors}. We shall show that there exists a $\bm{\Pi}_2\in \mathrm{Auto}(\mc{G})$ so that $G_{\pi_2}$ contains a cycle of length $2$.

Let $\mc{G}^\prime$ be obtained from $\mc{G}$ via the vertex permutation $\pi_1$. Now if the permutation graph $G_{\pi_1}$ has only one cycle $\mc{C}_{\pi_1}^1$, we show that the vertices $\{i,j,k\}$ are adjacent to the same set of vertices. Indeed, it follows from the symmetry of $\mc{G}$ that $\forall p\in \mc{N}_i(\mc{G})\setminus \{j,k\},$ the edge $(i,p)\in \mc{E}(\mc{G})$ implies that  
\begin{equation}
(\pi_1(i),\pi_1(p))=(j,p)\in \mc{E}(\mc{G}^\prime).
\end{equation} Thus, $p$ is also a neighbor of $j$ and if $i\sim k$ so is $j\sim k$, or i.e., $\mc{N}_i(\mc{G})= \mc{N}_j(\mc{G})$. Similarly, one can prove that $\mc{N}_i(\mc{G})= \mc{N}_j(\mc{G})= \mc{N}_k(\mc{G})$. That is, vertices $i,j$ and $k$ are adjacent to the same set of vertices, and the connectivity between them is either an empty or a cycle graph.  

On the other hand, for every vertex $p\in \mc{N}_i(\mc{G}) $ such that $p\not\in \mc{N}_j(\mc{G})$, then the condition
\begin{equation}
(\pi_1(i),\pi_1(p))=(j,\pi_1(p))\in \mc{E}(\mc{G}^\prime)
\end{equation}
indicates that the vertex $j\sim\pi_1(p)$ in $\mc{G}^\prime$. As a result, since $p\not\in \mc{N}_j(\mc{G})$ there exists $q\in \mc{N}_j(\mc{G})$ such that $q=\pi_1(p)$. That is, the vertex $p\in\mc{N}_i(\mc{G})$ must take the position of the vertex $q\in\mc{N}_j(\mc{G})$ under the permutation $\pi_1$ (see Fig. \ref{fig:three_vertices}). An analogous permutation of each $\kappa$-hop neighbor, $\kappa=2,\ldots$, of $i$ to the corresponding $\kappa$-hop neighbor of $j$, can also be obtained similarly. Consequently, the permutation graph $G_{\pi_1}$ has more than one permutation cycles including $\mc{C}_{\pi_1}^1$. Furthermore, by employing a similar argument for the permutation from $j$ to $k$ ($j\rightarrow k$), we conclude that the graph structures induced by the neighbors (including lower-hop neighbors if they exist) of the three vertices $i,j,$ and $j$, respectively, should be the same, and the permutation of them via $3$-length cycles leaves $\mc{G}$ unchanged, i.e., $\mc{G}= \mc{G}^\prime$.

As a result, in either case, we can simply permute, say vertex $i$ with $j$ ($i\leftrightarrows j$), and their $\kappa$-hop neighbors ($\kappa=1,2,\ldots$), with each other without changing the structure of $\mc{G}$. This constitutes a vertex mapping $\pi_2$ whose permutation graph contains $2$-length cycles.
\end{proof}

It is noted from the above result that a symmetric graph has $l$ subgraphs ($l\geq 2$), denoted as $\{\mc{G}_1^S,\ldots,\mc{G}_l^S\}$, formed by the respective disjoint subsets of one or more vertices, $\{\mc{V}_1^S,\ldots,\mc{V}_l^S\}$, of the same graph structure. The vertices in each pair of such subgraphs are connected to the remaining vertices, i.e., $\mc{V}\setminus \{\mc{V}_1^S,\ldots,\mc{V}_l^S\}$, in the same way. There may be also edges between pairs of corresponding vertices in the two subgraphs. For example, the network in Fig. \ref{fig:example_symmetric_graph} contains two subgraphs on vertices $\mc{V}^S_1=\{1,2,3\}$ and $\mc{V}^S_2=\{4,5,6\}$ of the same triangulation structure. These two subgraphs are connected to the same portion of the rest of the vertices in the graph. There may be also inter-edges between two subgraphs connecting some pairs of corresponding vertices, e.g., $(2,5)$ and $(1,4)$. We will refer to such subgraphs as \textit{subgraphs of symmetry} and summarize this result in the following.
\begin{Corollary}\label{coroll:copies_of_same_graph}
Any symmetric graph has two or more subgraphs of the same graph structure, which are adjacent to the same set of vertices in the rest of the graph.
\end{Corollary}
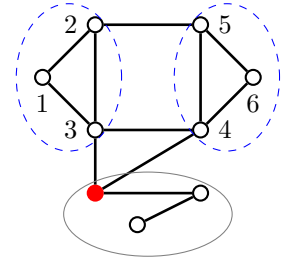
\begin{figure}[t]
\centering
\begin{tikzpicture}[scale=1.4]
\node[place] (4) at (1,0.) [label=right:$4$] {};
\node[place] (3) at (0,0) [label=left:$3$] {};
\node[place] (1) at (-0.5,.5) [label=below:$1$] {};
\node[place] (2) at (0,1.) [label=left:$2$] {};
\node[place] (5) at (1.,1.) [label=right:$5$] {};
\node[place] (6) at (1.5,0.5) [label=below:$6$] {};
\node[place, red] (7) at (0,-0.6) [] {};
\node[place] (8) at (1,-0.6) [] {};
\node[place] (9) at (0.4,-0.9) [] {};

\draw (2) [line width=1pt] -- node [left] {} (3) (1)--(3) (3)--(7);
\draw (3) [line width=1pt] -- node [left] {} (4) (7)--(8);
\draw (4) [line width=1pt] -- node [left] {} (5);
\draw (2) [line width=1pt] -- node [left] {} (5);
\draw (2) [line width=1pt] -- node [left] {} (1);
\draw (6) [line width=1pt] -- node [left] {} (5) 
	(6)--(4) 
	(4)--(7)
	(8)--(9);
\draw[dashed, rotate = 0, blue] (-0.25,0.5) ellipse (.5cm and 0.7cm);
\draw[dashed, rotate = 0, blue] (1.25,0.5) ellipse (.5cm and 0.7cm);
\draw[rotate = 0, gray] (0.5,-0.8) ellipse (.8cm and 0.4cm);
\end{tikzpicture}
\caption{Example of symmetric graph: two triangulation subgraphs (enclosed in dashed eclipses) are connected to the same (red) vertices in the rest of the graph.}
\label{fig:example_symmetric_graph}
\end{figure}
\subsection{Spectral properties of symmetric graphs}
We now can put together the results in Theorem \ref{thm:existence_cycle_length2} and Lem. \ref{lm:minus_one_eigenvalue} above to prove the following theorem.
\begin{Theorem}\label{thm:orthogonal_eigenvector}
If the adjacency matrix $\bm{A}$ of a symmetric graph $\mc{G}$ has simple spectrum, then there exists a nontrivial permutation $\bm{\Pi}\in \mathrm{Auto}(\mc{G})$ having an eigenvalue of $-1$ and $\bm{A}$ has an eigenvector orthogonal to $\bm{1}_n$.
\end{Theorem}
\begin{proof}
Since $\mc{G}$ is symmetric, there exists a nontrivial $\bm{\Pi}\in \mc{P}_n$ so that $\bm{A\Pi}=\bm{\Pi A}$. Now, we assume that $\bm{v}_i$ is an eigenvector of $\bm{A}$ corresponding to eigenvalue $\lambda_i$. It then follows that 
\begin{align*}
\bm{A\Pi}\bm{v}_i&=\bm{\Pi A}\bm{v}_i= \lambda_i \bm{\Pi}\bm{v}_i\\
\Longleftrightarrow  \bm{A}(\bm{\Pi}\bm{v}_i)&= \lambda_i (\bm{\Pi}\bm{v}_i) \numberthis.
\end{align*}
Therefore, $\bm{\Pi}\bm{v}_i$ is an eigenvector of $\bm{A}$. Suppose that $\bm{A}$ has distinct eigenvalues as otherwise the graph $\mc{G}(\bm{A})$ is unfriendly. If $\lambda_i$ is simple, one has either $\bm{\Pi}\bm{v}_i=\bm{v}_i$ or $\bm{\Pi}\bm{v}_i=-\bm{v}_i$. Consequently, $\bm{v}_i$ is also an eigenvector of $\bm{\Pi}$ (the converse is not necessarily true). It is not hard to see that $\bm{\Pi}$ always has an eigenvalue of $1$ corresponding to an eigenvector of the form \begin{equation}
\bm{u}_j=[\underbrace{1,1,1,\ldots,1}_{l},\underbrace{0,\ldots,0}_{n-l}]^\top,
\end{equation} associated with a permutation cycle $\mc{C}_\pi^j$ (of any length) and a re-labeling of the vertices in $\mc{V}$ so that $\mc{V}(\mc{C}_\pi^j)=\{1,2,\ldots,l\}$, where $l:=|\mc{V}(\mc{C}_\pi^j)|$ (see Proof of Lemma \ref{lm:minus_one_eigenvalue}). The number of the independent eigenvectors $\bm{u}_j$ corresponding to the eigenvalue $\lambda =1$ of a nontrivial $\bm{\Pi}\in \mc{P}_n$ equals to the number of independent permutation cycles in $G_\pi$, which is at most $n-1$ (zero length cycles, i.e., identical mapping, are counted). Since $\bm{A}$ has $n$ independent eigenvectors $\bm{v}_i$, there exist some $\bm{v}_i$ (corresponding to the eigenvalues $-1$ of $\bm{\Pi}$) satisfying $\bm{\Pi}\bm{v}_i=-\bm{v}_i$.
Moreover, Theorem \ref{thm:existence_cycle_length2} and Lemma \ref{lm:minus_one_eigenvalue} verify that $\bm{\Pi}$ indeed has an eigenvalue of $-1$. It follows that 
\begin{equation*}
\bm{1}_n^\top\bm{\Pi}\bm{v}_i=\bm{1}_n^\top\bm{v}_i=- \bm{1}_n^\top\bm{v}_i\Rightarrow \bm{1}_n^\top\bm{v}_i =0,
\end{equation*}
for some eigenvector $\bm{v}_i$. This completes the proof.
\end{proof}

The proof of Theorem \ref{thm:orthogonal_eigenvector} above also indicates that the number of eigenvectors of the adjacency matrix $\bm{A}$ orthogonal to $\bm{1}_n$ is equal to the maximum number of eigenvalues $-1$ of $\bm{\Pi}$ among all possible permutation matrices $\bm{\Pi}\in \mathrm{Auto}(\mc{G})$ (Corollary \ref{coroll:number_of_eig}).
\section{Unfriendliness and uncontrollability of multi-agent systems}\label{sec:relation_to_uncontroll}
This section examines the relation of graph symmetry to uncontrollability of multi-agent systems with a controlled leader under the consensus dynamics.
\subsection{Agreement dynamics with a leader}
Consider a system of $n$ agents with a connected graph $\mc{G}$. Suppose that each agent $i$ is associated with a state $x_i\in\mb{R}$. Let $\bm{x}=[x_1,x_2,\ldots,x_n]\in\mb{R}^n$ be the aggregated state vector of the system. Then, the well-known agreement dynamics of the system is given as \cite{Saber2004tac}
\begin{equation}\label{eq:agreement_dyn}
\dot{\bm{x}}=-\bm{L}(\mc{G})\bm{x}.
\end{equation}

Assume that there is a leader, say the last node $n$, whose state $x_n = u(t)$ for an exogenous control $u(t)\in \mb{R}$; the other agents, called \textit{followers}, obey the agreement dynamics \eqref{eq:agreement_dyn}. Denote $\mc{V}_l=\{n\}$ and $\mc{V}_f=\{1,\ldots,n_f\},n_f := n-1,$ as the leader and and the follower set, respectively. Define $-\bm{l}_{fl}=[\delta_{n}(1),\ldots,\delta_{n}(n_f)]\in \mb{R}^{n_f},$ where the \textit{indicator function} $\delta_n(i)=1$ if $i\sim n$, and $0$ otherwise. The \textit{follower-leader degree matrix} is defined as $\bm{D}_{fl}(\mc{G})=\mathrm{diag}(\{\delta_{n}(i)\}_{i=1}^{n_f})$. 

Let $\bm{x}_f=[x_1,x_2,\ldots,x_{n_f}]\in\mb{R}^{n_f}$; then the dynamics of the followers' state vector $\bm{x}_f$ can be obtained as the following controlled linear time-invariant system \cite{Rahmani2009siam}
\begin{equation}\label{eq:controlled_LTI}
\dot{\bm{x}}_f=-\bm{L}_f(\mc{G})\bm{x}_f-\bm{l}_{fl}u(t),
\end{equation}
where
$\bm{L}_f(\mc{G}):=\bm{L}(\mc{G}_f)+\bm{D}_{fl}(\mc{G})$ with $\bm{L}(\mc{G}_f)\in \mb{R}^{n_f\times n_f}$ being the Laplacian matrix of the graph of followers. 

\subsection{Sufficient condition for uncontrollability of \eqref{eq:controlled_LTI}}
A condition on the spectral properties, which may be called friendliness properties, of the matrix $\bm{L}_f(\mc{G})$ for controllability of \eqref{eq:controlled_LTI} is given as follows \cite{Rahmani2009siam}.
\begin{Lemma}\label{lm:controllable_cond}
\cite{Rahmani2009siam} The networked system \eqref{eq:controlled_LTI} with a single leader is controllable if and only if $\bm{L}_f(\mc{G})$ have no eigenvectors orthogonal to $\bm{1}_{n_f}$. Furthermore, if $\bm{L}_f(\mc{G})$ does not have simple spectrum, then \eqref{eq:controlled_LTI} is not controllable.
\end{Lemma} 

In the presence of the leader $n$, the system is called \textit{leader symmetric} if there exists a nontrivial permutation matrix $\bm{\Pi}\in \mc{P}_{n_f}$ that satisfies $\bm{\Pi}\bm{L}_f=\bm{L}_f\bm{\Pi}$ (this equality is similar to Eq. \eqref{eq:sym_cond_laplacian} above). This symmetry condition is equivalent to having a symmetric graph of followers $\mc{G}_f$, and for each $i\sim n$, it holds $\pi(i)\sim n$ \cite[Prop. 5.13]{Rahmani2009siam}. In addition, a sufficient graph-theoretic condition for the system \eqref{eq:controlled_LTI} to be uncontrollable is that $\mc{G}$ is leader symmetric. 

In the light of Corollary \ref{coroll:copies_of_same_graph}, we have the following result.
\begin{Lemma}
Consider the system \eqref{eq:controlled_LTI} of $n$ agents with a controlled leader $n$ whose follower graph $\mc{G}_f$ is symmetric. If the leader $n$ is adjacent to both vertices in some  pairs of corresponding vertices in two subgraphs of symmetry then the system \eqref{eq:controlled_LTI} is uncontrollable.
\end{Lemma}
\begin{proof}
Since graph $\mc{G}_f$ is symmetric, there exist two subgraphs of symmetry, say $\mc{G}_1^S$ and $\mc{G}_2^S$, in $\mc{G}_f$ (Coroll. \ref{coroll:copies_of_same_graph}).
Consider a vertex mapping $\pi$ that permutes the vertices in the two subgraphs $\mc{G}_1^S$ and $\mc{G}_2^S$ of the follower graph $\mc{G}_f$ by permutation cycles of length $2$. Since the leader $n$ is adjacent to some pairs of corresponding vertices in the two subgraphs, i.e., if $i\sim n$ and so is $\pi(i)\sim n$ for some vertex $i\in \mc{V}(\mc{G}_1^S)$ (thus $\pi(i)\in \mc{V}(\mc{G}_2^S)$), $\delta_{n}(i)=\delta_{n}(\pi(i))$. Therefore, the corresponding automorphism $\bm{\Pi}$ of $\mc{G}_f$ satisfies the following  
\begin{equation}
\bm{\Pi}\bm{L}_f=\bm{L}_f\bm{\Pi},
\end{equation}
where
$\bm{L}_f:=\bm{L}(\mc{G}_f)+\mathrm{diag}(\{\delta_{n}(i)\}_{i=1}^{n_f})$ with $\bm{L}(\mc{G}_f)\in \mb{R}^{n_f\times n_f}$ being the Laplacian matrix of the graph $\mc{G}_f$. It follows from the preceding equation and by following similar lines as in Proof of Theorem \ref{thm:orthogonal_eigenvector}, we have that $\bm{L}_f$ has either repeated eigenvalues or eigenvectors orthogonal to $\bm{1}_{n_f}$. As a result, the system \eqref{eq:controlled_LTI} is uncontrollable according to Lemma \ref{lm:controllable_cond}.
\end{proof}

For example, suppose that the graph $\mc{G}_f$ is given as in Fig. \ref{fig:example_symmetric_graph} with two triangulation subgraphs of symmetry. The system \eqref{eq:controlled_LTI} is then uncontrollable if the leader $n$ is connected to both vertices in any or combinations of the vertex pairs $\{2,5\},\{1,6\}$ and $\{3,4\}$.

\section{Examples}\label{sec:examples}
\subsection{Simple symmetric graph of $6$ vertices}
 \begin{figure*}[t]
\centering
\begin{subfigure}[b]{0.2\textwidth}
\centering
\begin{tikzpicture}[scale=1.1]
\node[place] (4) at (1,0.) [label=right:$4$] {};
\node[place] (3) at (0,0) [label=left:$3$] {};
\node[place] (1) at (0,1.6) [label=left:$1$] {};
\node[place] (2) at (0,1.) [label=left:$2$] {};
\node[place] (5) at (1.,1.) [label=right:$5$] {};
\node[place] (6) at (1.,1.6) [label=right:$6$] {};

\draw (2) [line width=1pt] -- node [left] {} (3);
\draw (3) [line width=1pt] -- node [left] {} (4);
\draw (4) [line width=1pt] -- node [left] {} (5);
\draw (2) [line width=1pt] -- node [left] {} (5);
\draw (2) [line width=1pt] -- node [left] {} (1);
\draw (6) [line width=1pt] -- node [left] {} (5);
\draw[dashed, rotate = 0, gray] (0,0.8) ellipse (.18cm and 1.1cm);
\draw[dashed, rotate = 0, gray] (1,0.8) ellipse (.18cm and 1.1cm);
\end{tikzpicture}
\caption{$\mc{G}_f$}
\label{fig:graph_6}
\end{subfigure}
\begin{subfigure}[b]{0.25\textwidth}
\centering
\begin{tikzpicture}
 [>=stealth,
   shorten >=1pt,
   node distance=2cm,
   on grid,
   auto,
   every state/.style={draw=black, fill=black!6, very thick},scale =1.4
  ]
\tikzstyle{neuron}=[place,circle,inner sep=0,minimum size=10pt]
\node[neuron] (4) at (1,0.) [label=right:] {4};
\node[neuron] (3) at (0,0) [label=left:] {3};
\node[neuron] (2) at (0,1) [label=left:] {2};
\node[neuron] (1) at (-1,0) [label=left:] {1};
\node[neuron] (6) at (-0.7,1.) [] {6};
\node[neuron] (5) at (1.,1.) [] {5};
\path[->] 
   (1) edge[bend left,blue]     node                   {} (6)
   (6) edge[bend left,blue]     node                   {} (1)
   (3) edge[bend right,blue]     node                   {} (4)
   (4) edge[bend right,blue]     node                   {} (3)
   (5) edge[bend right,blue]     node                   {} (2)
   (2) edge[bend right,blue]     node                   {} (5);
\end{tikzpicture}
\caption{$G_\pi$}
\end{subfigure}
\begin{subfigure}[b]{0.4\textwidth}
\centering
\begin{equation*}
\bm{U}=
\left[\begin{smallmatrix}
-0.2319   &-0.5211    &0.4179   &-0.4179   &-0.5211    &0.2319\\
    0.5211    &0.4179   &-0.2319   &-0.2319   &-0.4179    &0.5211\\
   -0.4179   &-0.2319   &-0.5211    &0.5211   &-0.2319    &0.4179\\
    0.4179   &-0.2319    &0.5211    &0.5211   & 0.2319    &0.4179\\
   -0.5211    &0.4179    &0.2319   &-0.2319   & 0.4179    &0.5211\\
    0.2319   &-0.5211   &-0.4179   &-0.4179   & 0.5211    &0.2319
\end{smallmatrix}\right]
\end{equation*}
\caption{}
\end{subfigure}\\
\begin{subfigure}[t]{0.82\textwidth}
\centering
\begin{align*}
\lambda(\bm{A})=\{-2.24,-0.8019,-0.555,0.555,0.8019,-2.24\},~
\bm{U}^\top\bm{1}_6=[0,
   -0.6703,
    0,
   -0.2574,
    0,
    2.3419]^\top
\end{align*}
\caption{}
\end{subfigure}
\caption{A symmetric graph of $6$ vertices. $\bm{A}(\mc{G}_f)$ has simple spectrum, but it has three eigenvectors, given as columns of $\bm{U}$ (such that $\bm{A}=\bm{U}\mathrm{diag}(\lambda)\bm{U}^\top$), orthogonal to $\bm{1}_6$. Under the permutation $G_\pi$ the graph topology of $\mc{G}_f$ remains unchanged.}
\label{fig:sym_graph_6}
\end{figure*}

Consider a symmetric graph $\mc{G}_f$ of six vertices, as given in Fig. \ref{fig:graph_6}.
The adjacency matrix of the symmetric graph has distinct eigenvalues, but has three eigenvectors orthogonal to $\bm{1}_6$ (Fig. \ref{fig:sym_graph_6}) since there are three independent cycles of length $2$ in $G_\pi$ (Corollary \ref{coroll:number_of_eig}).

When connecting the leader node $7$ to the first node in $\mc{G}_f$, i.e., $\bm{l}_{fl}=[1,0,0,0,0,0]^\top$, the matrix $\bm{L}_f(\mc{G}):=\bm{L}(\mc{G}_f)+\mathrm{diag}(\bm{l}_{fl})$ given explicitly in \eqref{eq:L_f_controllable} below has distinct eigenvalues and no eigenvectors orthogonal to $\bm{1}_6$. Thus, the system \eqref{eq:controlled_LTI} is controllable.
\begin{equation}\label{eq:L_f_controllable}
\bm{L}_f(\mc{G})=\left[\begin{matrix}
2 &-1 &0 &0 &0 &0\\
-1 &3 &-1 &0 &-1 &0\\
0 &-1 &2 &-1 &0 &0\\
0 &0 &-1 &2 &-1 &0\\
0 &-1 &0 &-1 &3 &-1\\
0 &0 &0 &0 &-1 &1
\end{matrix}\right].
\end{equation}
In contrast, if the leader $7$ is connected to both nodes $1$ and $6$, the matrix $\bm{L}_f(\mc{G})$ has three eigenvectors orthogonal to $\bm{1}_6$. Consequently, the system \eqref{eq:controlled_LTI} is not controllable.
\subsection{Contiguous USA graph}
The "\textit{contiguous USA graph}" is the graph whose vertices represent the contiguous $48$ states of the United States plus the District of Columbia (DC) and whose ($107$) edges connect pairs of states (plus DC) that are connected by at least one drivable road \cite{Knuth2008}. The graph consists of $49$ vertices (red nodes) and $107$ edges (solid black lines) as shown in Fig. \ref{fig:US_Graph}. It can be verified by numerical computation that the adjacency matrix of the contiguous USA graph has distinct eigenvalues and no eigenvectors orthogonal to the vector of all ones. Thus, the graph is asymmetric, as it can also be seen from Fig. \ref{fig:US_Graph}.
\begin{figure}[t]
\centering
\includegraphics[width=0.49\textwidth]{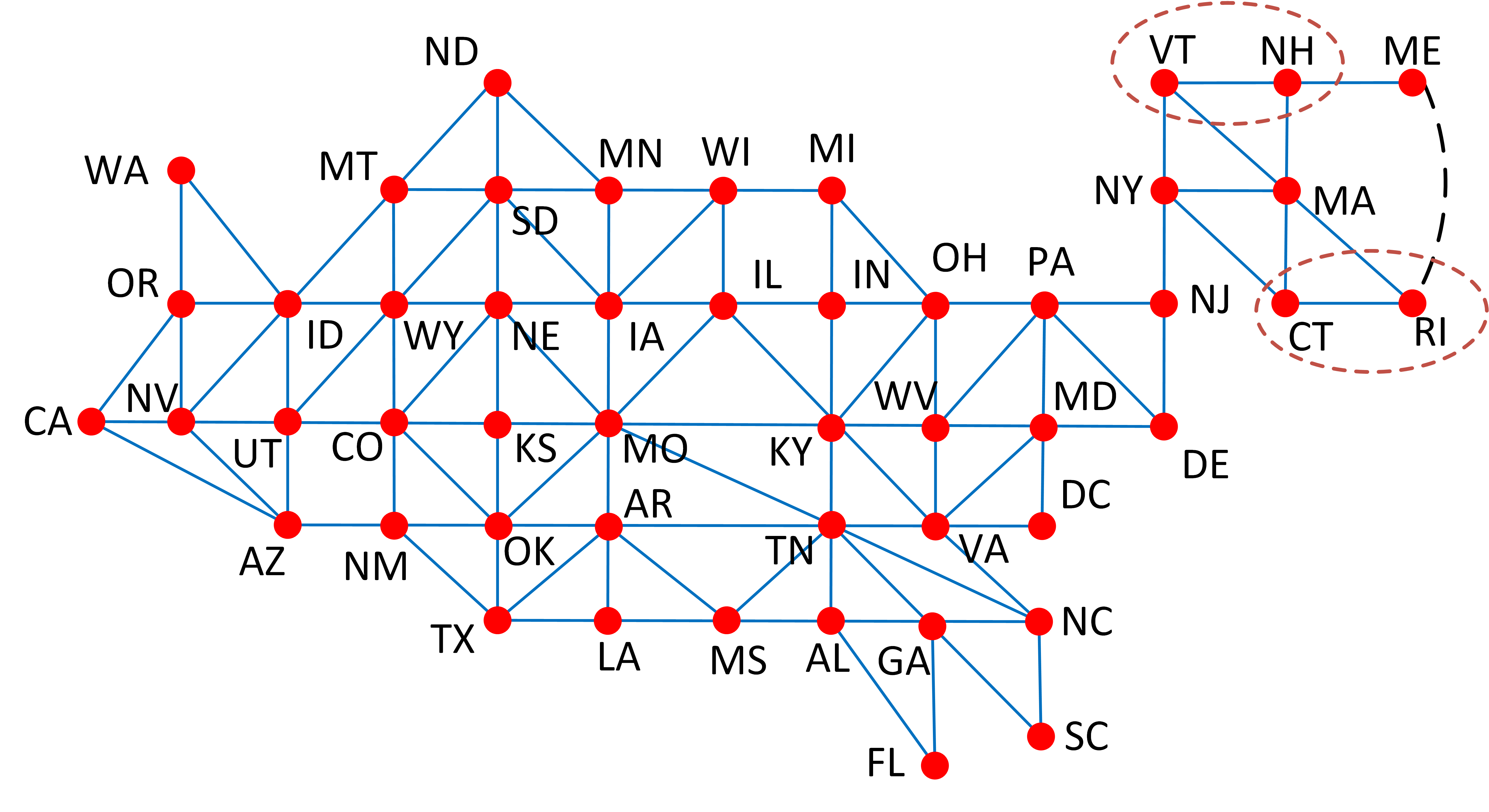}
\caption{Contiguous USA graph. An edge (dashed curve) is added between ME and RI to make the graph symmetric.}
\label{fig:US_Graph}
\end{figure}

To make the graph symmetric, we in addition insert an edge between nodes ME and RI. Due to the edge insertion, two subgraphs (enclosed in dashed eclipses) associated with $\{\text{VT},\text{NH}\}$ and $\{\text{CT},\text{RI}\}$ now have the same line graph structure and adjacent in the same way to the vertices $\{\text{NY, MA, ME}\}$ (cf. Fig. \ref{fig:US_Graph}). As a result, the permutation $\pi$ that exchanges the positions of the vertices $\{\text{VT},\text{NH}\}$ with $\{\text{CT},\text{RI}\}$ leaves the new graph topology unchanged. The permutation graph of such a permutation contains two cycles of length two. The adjacency matrix of the graph with the newly added edge has distinct eigenvalues, but has two eigenvectors orthogonal to $\bm{1}_{49}$.
\section{Conclusion}\label{sec:conclusion}
Unfriendliness of symmetric graphs was revisited in this paper. We first presented the spectral properties of permutation matrices through identification of the corresponding permutation graph representations. Second, symmetric structures of graphs in association with the permutation graphs of vertex permutations were explored. In particular, it was shown that symmetric graphs have two or more subgraphs that have the same graph topology and are adjacent to the same set of vertices. We then proved that if adjacency matrices of symmetric graphs have simple spectrum then they have eigenvectors orthogonal to the vector of all ones, thus possessing unfriendly properties. Several examples of both synthetic and real-world graphs were also given to illustrate the results.
\section*{Acknowledgments}
This work was supported by the National Research Foundation (NRF) of Korea under the grant NRF-2022R1A2B5B03001459, in part by the BK21 FOUR Program of the National Research Foundation Korea (NRF) grant funded by the Ministry of Education(MOE), and in part by the Future Mobility Testbed Development through IT, AI, and Robotics.
\nocite{} 
\bibliographystyle{IEEEtran}
\bibliography{IEEEabrv,quoc2018,quoc2019} 

\begin{thebibliography}{10}
\providecommand{\url}[1]{#1}
\csname url@samestyle\endcsname
\providecommand{\newblock}{\relax}
\providecommand{\bibinfo}[2]{#2}
\providecommand{\BIBentrySTDinterwordspacing}{\spaceskip=0pt\relax}
\providecommand{\BIBentryALTinterwordstretchfactor}{4}
\providecommand{\BIBentryALTinterwordspacing}{\spaceskip=\fontdimen2\font plus
\BIBentryALTinterwordstretchfactor\fontdimen3\font minus
  \fontdimen4\font\relax}
\providecommand{\BIBforeignlanguage}[2]{{%
\expandafter\ifx\csname l@#1\endcsname\relax
\typeout{** WARNING: IEEEtran.bst: No hyphenation pattern has been}%
\typeout{** loaded for the language `#1'. Using the pattern for}%
\typeout{** the default language instead.}%
\else
\language=\csname l@#1\endcsname
\fi
#2}}
\providecommand{\BIBdecl}{\relax}
\BIBdecl

\bibitem{Friedkin2016Sci}
N.~E. Friedkin, A.~V. Proskurnikov, R.~Tempo, and S.~E. Parsegov, ``Network
  science on belief system dynamics under logic constraints,'' \emph{Science},
  vol. 354, no. 6310, pp. 321--326, 2016.

\bibitem{Quoc2018tcns}
Q.~V. Tran, M.~H. Trinh, D.~Zelazo, D.~Mukherjee, and H.-S. Ahn, ``Finite-time
  bearing-only formation control via distributed global orientation
  estimation,'' \emph{IEEE Trans. Control Network Syst.}, vol.~2, no.~6, pp.
  702--712, 2019.

\bibitem{Bullo2020}
\BIBentryALTinterwordspacing
F.~Bullo, \emph{Lectures on Network Systems}, 1st~ed.\hskip 1em plus 0.5em
  minus 0.4em\relax Kindle Direct Publishing, 2020, with contributions by J.
  Cortes, F. Dorfler, and S. Martinez. [Online]. Available:
  \url{http://motion.me.ucsb.edu/book-lns}
\BIBentrySTDinterwordspacing

\bibitem{Quoc2020tcns}
Q.~V. Tran and H.-S. Ahn, ``Distributed formation control of mobile agents via
  global orientation estimation,'' \emph{IEEE Trans. Control Network Syst.},
  vol.~4, no.~7, pp. 1654--1664, 2020.

\bibitem{Mesbahi2010}
M.~Mesbahi and M.~Egerstedt, \emph{Graph Theoretic Methods in Multiagent
  Networks}.\hskip 1em plus 0.5em minus 0.4em\relax Princeton University Press,
  2010.

\bibitem{HSAhn2019}
H.-S. Ahn, \emph{Formation Control: Approaches for Distributed Agents}.\hskip
  1em plus 0.5em minus 0.4em\relax Springer International Publishing, 2019.

\bibitem{Quoc2020auto}
Q.~V. Tran, B.~D.~O. Anderson, and H.-S. Ahn, ``Pose localization of
  leader-follower networks with direction measurements,'' \emph{Automatica},
  vol. 120, p. 109125, 2020.

\bibitem{Quoc2021tcns}
Q.~V. {Tran}, M.~H. Trinh, and H.-S. {Ahn}, ``Discrete-time matrix weighted
  consensus,'' \emph{IEEE Trans. Control Network Syst.}, pp. 1--11, 2021.

\bibitem{DefferrardNIPS16}
M.~Defferrard, X.~Bresson, and P.~Vandergheynst, ``Convolutional neural
  networks on graphs with fast localized spectral filtering,'' in \emph{Adv.
  Neural Inf. Process. Syst.}, D.~Lee, M.~Sugiyama, U.~Luxburg, I.~Guyon, and
  R.~Garnett, Eds., vol.~29, 2016.

\bibitem{SeifertTSP21}
B.~Seifert and M.~Püschel, ``Digraph signal processing with generalized
  boundary conditions,'' \emph{IEEE Trans. Signal Process.}, vol.~69, pp.
  1422--1437, 2021.

\bibitem{Foggia2014}
P.~Foggia, G.~Percannella, and M.~Vento, ``Graph matching and learning in
  pattern recognition in the last $10$ years,'' \emph{Int. J. Pattern Recogn.
  Artif. Intell.}, vol.~28, no.~1, pp. 1\,450\,001--1--40, 2014.

\bibitem{Vogelstein2015PloSONE}
J.~T. Vogelstein, J.~M. Conroy, V.~Lyzinski, L.~J. Podrazik, S.~G. Kratzer,
  E.~T. Harley, D.~E. Fishkind, R.~J. Vogelstein, and C.~E. Priebe, ``Fast
  approximate quadratic programming for graph matching,'' \emph{PLoS ONE},
  vol.~10, no.~4, 2015, p. e0121002.

\bibitem{QuocCDC20}
Q.~V. Tran, Z.~Sun, B.~D.~O. Anderson, and H.-S. Ahn, ``Distributed computation
  of graph matching in multi-agent networks,'' in \emph{Proc. the 59th IEEE
  Confer. Decision Control (CDC)}.\hskip 1em plus 0.5em minus 0.4em\relax Proc.
  the 59th IEEE Confer. Decision Control (CDC), 2020, pp. 3139--3144.

\bibitem{Aflalo2015pnas}
Y.~Aflalo, A.~Bronstein, and R.~Kimmel, ``On convex relaxation of graph
  isomorphism,'' \emph{Proceedings of the National Academy of Sciences}, vol.
  112, no.~1, pp. 2942--2947, 2015.

\bibitem{Erdos1963}
P.~Erd{\~{o}}s and A.~R\'{e}nyi, ``Asymmetric graphs,'' \emph{Acta Math. Acad.
  Sci. Hungar}, vol.~14, pp. 295--315, 1963.

\bibitem{Ullmann2011}
J.~R. Ullmann, ``Bit-vector algorithms for binary constraint satisfaction and
  subgraph isomorphism,'' \emph{J. Exp. Algorithmics}, vol.~15, pp.
  1.6:1.1--64, 2011.

\bibitem{Caelli2004pami}
T.~Caelli and S.~Kosinov, ``An eigenspace projection clustering method for
  inexact graph matching,'' \emph{IEEE Trans. Pattern Anal. Mach. Intell.},
  vol.~26, no.~4, pp. 515--519, 2004.

\bibitem{Duchenne2011pami}
O.~Duchenne, F.~Bach, I.-S. Kweon, and J.~Ponce, ``A tensor-based algorithm for
  high order graph matching,'' \emph{IEEE Trans. Pattern Anal. Mach. Intell.},
  vol.~33, no.~12, pp. 2383--2395, 2011.

\bibitem{ZFan2020icml}
Z.~Fan, C.~Mao, Y.~Wu, and J.~Xu, ``Spectral graph matching and regularized
  quadratic relaxations: {A}lgorithm and theory,'' in \emph{Proc. the 37th IEEE
  Confer. Machine Learning (ICML)}, 2020.

\bibitem{Lyzinski2016pami}
V.~Lyzinski, D.~E. Fishkind, M.~Fiori, J.~T. Vogelstein, C.~E. Priebe, and
  G.~Sapiro, ``Graph matching: Relax at your own risk,'' \emph{IEEE Trans.
  Pattern Anal. Mach. Intell.}, vol.~38, no.~1, pp. 60--73, 2016.

\bibitem{Rahmani2009siam}
A.~Rahmani, M.~Ji, M.~Mesbahi, and M.~Egerstedt, ``Controllability of
  multi-agent systems from a graph-theoretic perspective,'' \emph{SIAM J.
  Control Optim.}, vol.~48, no.~1, pp. 162--186, 2009.

\bibitem{Quoc2022Tcyb}
Q.~V. Tran, Z.~Sun, B.~D.~O. Anderson, and H.-S. Ahn, ``Distributed
  optimization for graph matching,'' \emph{IEEE Trans. Cybern.}, pp. 1--14,
  2022.

\bibitem{Tao2017}
T.~Tao and V.~Vu, ``Random matrices have simple spectrum,''
  \emph{Combinatorica}, vol.~37, no.~3, pp. 539--553, 2017.

\bibitem{Hofstad2016}
R.~Hofstad, \emph{Random Graphs and Complex Networks: Volume 1}.\hskip 1em plus
  0.5em minus 0.4em\relax Cambridge, UK: Cambridge University Press, 2016.

\bibitem{Saber2004tac}
R.~Olfati-Saber and R.~M. Murray, ``Consensus problems in networks of agents
  with switching topology and time-delays,'' \emph{IEEE Trans. Autom. Control},
  vol.~49, no.~9, pp. 1520--1533, 2004.

\bibitem{Knuth2008}
D.~E. Knuth, \emph{The Art of Computer Programming, Volume 4, Fascicle 0:
  Introduction to Combinatorial Algorithms and Boolean Functions (Art of
  Computer Programming)}.\hskip 1em plus 0.5em minus 0.4em\relax Addison-Wesley
  Professional, 01 2008.

\end{thebibliography}

%
%
\end{document}